\documentclass[10pt]{amsart}
\usepackage{amsmath}
\usepackage{amssymb, amsfonts, amsmath, amsthm, amscd}
\usepackage{mathrsfs}
\usepackage{color, graphics}
\usepackage[all]{xy}
\usepackage{hyperref}
\usepackage{amsxtra} 

\newtheorem{theorem}{Theorem}[section]

\newtheorem{claim}[theorem]{Claim}

\theoremstyle{remark}

\numberwithin{equation}{section}

\def\dim{\operatorname{dim}}%
\def\dim{\operatorname{dim}}%

\def\Ext{\operatorname{Ext}}

\def\Oc{\mathcal O}

\def\cL{\mathcal L}

\def\cU{\mathcal U}
\def\cG{\mathcal G}
\def\cF{\mathcal F}
\def\cH{\mathcal H}
\newcommand{\sA}{\mathscr{A}}
\newcommand{\sB}{\mathscr{B}}
\def\FF{\mathbb F}

\def\E{\mathcal E}
\def\Pp{\mathbb P}

\begin{document}

\title[Ulrich wildness of some threefold scrolls]{Ulrich wildness of some decomposable threefold scrolls over $\mathbb F_a$}

\author{Maria Lucia Fania}
\address{Dipartimento di Ingegneria e Scienze dell'Informazione e Matematica, Universit\`{a} degli Studi di L'Aquila,  Via Vetoio Loc. Coppito, 67100, L'Aquila, Italy}
\email{marialucia.fania@univaq.it}

\author{Flaminio Flamini}
\address{Dipartimento di Matematica, Universita' degli Studi di Roma Tor Vergata, Via della Ricerca Scientifica-00133 Roma, Italy}
\email{flamini@mat.uniroma2.it}



\begin{abstract} The paper deals with {\em Ulrich wildness} of decomposable threefold scrolls $X$ over Hirzebruch surfaces $\mathbb{F}_a$, for any $a \geqslant 0$. Our {\em Main Theorem} enstablishes that for $a=0$, the moduli space of rank-$r$ Ulrich bundles, for any $r \geqslant 2$ and of given Chern classes, contains a {\em generically smooth}, {\em unirational} component $\mathcal{M}(r)$ of computed dimension whose general point corresponds to a {\em slope-stable} Ulrich bundle; in particular $X$ turns out to be {\em Ulrich wild}. When $a \geqslant 1$ and in presence of {\em modular obstructions}, $X$ is nevertheless shown to be {\em Ulrich wild}.
\end{abstract}

\thanks{The authors would like to deeeply thank the organizers (F. Catanese and J.H. Keum) of the {\em ``46th Frontier Scientists Workshop: New trends in Algebraic and Complex Geometry"}, Ischia (Na)-Italy, 18-20 October 2025, in the framework of Accademia Nazionale dei Lincei and Korean Academy of Science and Technology, for the invitation of the second author as speaker at the workshop. The authors also thank F. Malaspina and J. Pons-Llopis for some useful conversations. The second author is member of  INdAM-GNSAGA. The second author has been partially supported by MIUR Excellence Department Project MatMod@TOV, MIUR CUP-E83C23000330006, 2023-2027,  awarded to the Department of Mathematics, University of Rome Tor Vergata.}

\maketitle

\section*{Introduction} Let $X \subset \mathbb{P}^N$ be a $n$-dimensional smooth projective variety with a hyperplane section divisor $H$. A vector bundle $\mathcal{U}$ on $X$ is called a \emph{$\mathcal O_X(H)$-Ulrich bundle} if $$H^*(\mathcal{U} \otimes \mathcal{O}_X(-iH)) = 0, \;\; 1 \leqslant i \leqslant n.$$The \emph{Ulrich complexity} of $X$ is defined as $$uc_H(X) := \min \{ r \in \mathbb{N}^* \mid \exists \, \mathcal{U} \text{ indecomposable Ulrich bundle, } \text{rk}(\mathcal{U})=r \}.$$Such bundles exhibit several key structural properties: they are closed under extensions; any such $\cU$ has an {\em Ulrich dual} bundle, $\mathcal{U}^U := \mathcal{U}^{\vee} \otimes \omega_X \otimes \mathcal{O}_X((n+1)H)$, of identical rank; moreover they are always ({\em Gieseker/slope}) {\em semistable}, and also {\em stable} if they are not extensions of lower-rank Ulrich bundles (cf. e.g. \cite{b,c-h-g-s}), so they fill up subsets of moduli spaces of semistable bundles. 

While well-characterized for curves and certain classes of surfaces, existence and rank classification remain largely open in dimension $n\geqslant 3$ (cf. e.g. \cite{CMP} for an overview).

The paper investigates Ulrich bundles on threefold scrolls over Hirzebruch surfaces $\mathbb{F}_a$, $a \geqslant 0$, $f$ denoting fiber class of $\FF_a$ whereas $C_{-}$ either the unique section of $\FF_a$ with $C_{-}^2 = -a$, if $a>0$, or the fiber of the other ruling, if $a=0$; so we denote line bundles on $\mathbb{F}_a$ as $\mathcal{O}_{\mathbb{F}_a}(\alpha, \beta) := \alpha C_{-} + \beta f$, $\alpha, \beta \in \mathbb Z$. 

In this set-up, taking the projective bundle $\varphi: \mathbb{P}_{\mathbb{F}_a}(\mathcal{E}_a^b) \to \mathbb{F}_a$, with $a, b \geqslant 0$ and $\mathcal{E}_a^b := \mathcal{O}_{\mathbb{F}_a} \oplus \mathcal{O}_{\mathbb{F}_a}(0, -b)$, let $\xi$ be the {\em tautological class}, $C_0 := \varphi^*(C_{-})$ and $F := \varphi^*(f)$, so that the {\em Chow ring} of $\mathbb{P}_{\mathbb{F}_a}(\mathcal{E}_a^b)$ is determined by $$\xi^2 = -b \xi F,\; C_0^2 = -a C_0 F, \; F^2=0\; {\rm and}\; \xi C_0 F=1.$$To ease notation, we also set:
\[\mathcal{O}_a(\alpha_1, \alpha_2, \alpha_3) := \alpha_1 \xi + \alpha_2 C_0 + \alpha_3 F \quad \text{and} \quad [\alpha_1, \alpha_2, \alpha_3] := \alpha_1 \xi C_0 + \alpha_2 \xi F + \alpha_3 C_0 F.
\] For any $c \in \mathbb{Z}$, we let $h = h_a := \mathcal{O}_a(1,1,c)$ which turns out to be very ample if and only if $c \geqslant a+b+1$ (cf. \cite[Remark\,2.1]{fa-fl-ma-pl}). Under this condition, one has the embedding:
\begin{equation}\label{eq:Xe}
    \Phi_{|h|}: \mathbb{P}_{\mathbb{F}_a}(\mathcal{E}_a^b) \hookrightarrow X \subset \mathbb{P}^N
\end{equation} so that $X$ is a {\em threefold scroll} over $\FF_a$ of degree $d = 3(2c-a-b)$, 
{\em canonical bundle} $\omega_X \cong \mathcal{O}_a(-2, -2, -(a+b+2))$, {\em sectional genus} $g = 2c - a - b - 1$, where $N:= h^0(\Pp_{\FF_a}(\E^b_a), h) -1 = 4c-2a-2b+3$.

Previous paper \cite{fa-fl-ma-pl} established a {\em multiple scroll structure} on $X$ and classified $h$-Ulrich line bundles on it, as follows. 

\bigskip

\noindent
\textsc{Theorem A} (cf. \cite[Prop.\,2.2,\,Cor.\,2.4,\,Thm.\;3.1]{fa-fl-ma-pl} ) {\em (1) For any integers $a,b \geqslant 0$ one has 
$$\Pp_{\FF_a}( \mathcal E_a^b = \Oc_{\FF_a} \oplus\Oc_{\FF_a}(0, -b)) \cong   \Pp_{\FF_{b}}(\mathcal E_b^a:= \Oc_{\FF_{b}}\oplus \Oc_{\FF_{b}}(0,-a)).$$In particular, for any $c \geqslant a+b+1$, threefold scrolls $X$ as in \eqref{eq:Xe} are endowed with a  
{\em multiple scroll structure} over both Hirzebruch surfaces $\FF_a$ and $\FF_b$ for which 
\begin{equation}\label{eq:motiva}
\mathcal O_a(1,1,c) = h_a = h = h_b = \mathcal O_b(1,1,c).
\end{equation}

\bigskip

\noindent
(2) For any $a,b \geqslant 0$ and $c \geqslant a+b+1$, the {\em Ulrich complexity} of $X$ is $uc_h(X)=1$. More precisely, isomorphism in (1) and simmetry in \eqref{eq:motiva} permit to assume $b\geqslant a \geqslant 0$ and to classify $h$-Ulrich line bundles on $X$ as follows: 

\medskip

\noindent
(i) for any $b \geqslant a > 0, \; c \geqslant a+b+1$, $X$ supports exactly two $h$-Ulrich line bundles, namely $N= \Oc_a(2,0,2c-a-1)$ and its Ulrich dual $N^U= \Oc_a (0,2,2c-b-1)$;

\medskip

\noindent
(ii) for any $b>0 =a, \; c \geqslant b+1$, $X$ supports exactly four $h$-Ulrich line bundles, namely $N$ and $N^U$   as in (i),  together with the two extra line bundles $L= \mathcal O_0(1,0,3c-b-1)$ and its Ulrich dual $L^U=  \Oc_0(1, 2,c-1)$;

\medskip

\noindent
(iii) if $a=b=0$ then, for any $c \geqslant 1$, $X$ supports exactly six $h$-Ulrich line bundles, namely $N, \; N^U, \; L, \; L^U$ as in (ii), together with the two extra line bundles $M=\Oc_0(2,1,c-1)$ and its Ulrich dual $M^U=\Oc_0(0,1,3c-1)$.
}

\bigskip

Inspired by representation theory of ACM-graded modules (cf. e.g. \cite{DG}), a smooth, projective variety $Y \subset \mathbb P^N$ is said to be {\em (geometrically)} \emph{$\mathcal O_Y(H)$-Ulrich wild} if it supports $p$-dimensional families of pairwise non-isomorphic, indecomposable $\mathcal O_Y(H)$-Ulrich bundles for arbitrarily large $p$; this property suggests a rich internal geometry of $Y$ (in particular the {\em Ulrich set} $Ur(Y)$ is large). In this circle of ideas, in \cite{fa-fl-ma-pl} the authors focused on threefold scrolls $(X, h)$ as above and considered modular components of $h$-Ulrich bundles arising from {\em iterated extensions} of the pair $(N, N^U)$, under $0 \leqslant a \leqslant b \leqslant 1$.

In fact, building on \textsc{Theorem A}-(2), \cite{fa-fl-ma-pl} first identified $h$-Ulrich line bundles $L$ and $L^U$ as  those obtained via {\em twisted pullbacks} from the base, namely of the form $h \otimes \varphi^*(\mathcal{L}_{\mathbb{F}_a})$ for some $\mathcal L_{\FF_a}$ which, up to a twist, is Ulrich on $\FF_a$ w.r.t. a natural very ample polarization on $\FF_a$ (cf. \cite[Thm.\,1.12,\;Rem.\,3.3]{fa-fl-ma-pl}). Instead, iterated extensions involving $(M, M^U)$ or even {\em mixed pairs} have been excluded in \cite{fa-fl-ma-pl}, as they have been proved to  yield only strictly semistable bundles, i.e. failing to produce positive-dimensional modular components (cf. Theorem 4.1-{\bf Cases k}, with $k=3,4,8$ in \cite{fa-fl-ma-pl}).

This explains why \cite[Thm.\,5.6]{fa-fl-ma-pl} focused on ({\em deformations} of) {\em iterated extensions} involving $(N, N^U)$, as these line bundles were proved to be not {\em twisted pullbacks} from either $\mathbb{F}_a$ or $\mathbb{F}_b$, dealing therefore with {\em $h$-Ulrichness} on $X$ not coming from base-surfaces. Nevertheless, in \cite[Thm.\,5.6]{fa-fl-ma-pl}, the authors were forced to use $0 \leqslant a \leqslant b \leqslant 1$ to avoid {\em obstructions} to the existence of generically smooth modular components for Ulrich bundles arising from them (cf. \cite[Thm.\,4.1-($1_b$),\; Thm.\,4.2-(jj),\,(jjj)]{fa-fl-ma-pl}).

\bigskip

The aim of this paper is to provide description of what occurs in the remaining  cases, distinguishing between two primary regimes based on the nature of the $h$-Ulrich line bundles and the modular behavior of their higher-rank iterated (deformed) extensions:
 
\medskip

\noindent
$\bullet$ {\bf The Inherited Case} $(L, L^U)$: for $b \geqslant 0 =a, \; c \geqslant b+1$, we consider here deformations of iterated extensions of $(L, L^U)$, both {\em twisted pullbacks} from the base $\mathbb{F}_0$.

\medskip

\noindent
$\bullet$ {\bf The Obstructed Case} $(N, N^U)$: outside the range $0 \leqslant a \leqslant b \leqslant 1$, the pair $(N, N^U)$ presents {\em cohomological obstructions} to the existence of smooth modular families (cf. \cite[Thm.\,4.1-($1_b$),\; Thm.\,4.2-(jj) and (jjj)]{fa-fl-ma-pl}). However, {\em $h$-Ulrich wildness} can still be proven without existence of smooth modular components.

\bigskip 

\noindent
\textsc{Main Theorem} {\em (1) Let $b \geqslant 0=a$ and $c \geqslant b+1$ be integers and let $(X, h)$ be a threefold  scroll over $\FF_0$ as in \eqref{eq:Xe}. Then, for any integer $r \geqslant 1$, the moduli space of rank-$r$ vector bundles $\cU_r$ on $X$, which are $h$-Ulrich and with Chern classes

   \begin{eqnarray} \label{eq:chernLLU}
   c_1(\cU_r): =
    \begin{cases} 
      \mathcal O_0 \left(r, \, r+1, \frac{(r-1)(4c-b-2)+ 2(c-1)}{2}\right), & \mbox{if $r$ is odd}, \\ 
      & \\
      \mathcal O_0 \left(r, \, r, \frac{r}{2}(4c-b-2)\right), & \mbox{if $r$ is even},  
    \end{cases}
  \end{eqnarray}
     \begin{eqnarray*}
{\small c_2(\cU_r) =
    \begin{cases} 
    \scriptstyle \left[r^2-1, \; r^2 (2c-b-1) + r - (3r-1) c + b \frac{(3r-1)}{2}, \; (r^2-1)(2c-1) - b \frac{(r^2-1)}{2}  \right], & \mbox{if $r\geqslant 3$ odd}, \\
    & \\
\scriptstyle \left[ r(r-1), \; r(r-1) (2c-b-1), \; \frac{r^2}{2}(4c-b-2) +(1-c)r\right], & \mbox {if $r$ even},  
    \end{cases} }
     \end{eqnarray*}    
     \begin{eqnarray*}
     {\small 
 c_3(\cU_r) = \begin{cases} (2c-b-1)r^3+(2b-4c+2)r^2- (2c-b-1)r-2b+4c-2, & \mbox{if $r\geqslant 3$ odd}, \\
& \\  (2c-b-1)r^3+5b \frac{r^2}{2} - (5c-3)r^2+ (2c-b-2)r, & \mbox{if $r \geqslant 4$  even},  
    \end{cases}}
      \end{eqnarray*} is not empty and it contains a generically smooth component, denoted by $\mathcal M(r)$, which is moreover {\em unirational} for $r \geqslant 3$ and {\em rational} for $r=2$, of dimension 
		\begin{eqnarray}\label{eq:dimLLU}
		\dim (\mathcal M(r) ) = \begin{cases} (r^2 -1)(2c-b-1), & \mbox{if $r$ is odd}, \\
			 r^2(2c-b-1) +1 , & \mbox{if $r$ is even},
    \end{cases}
    \end{eqnarray} whose general point $[\cU_r] \in \mathcal M(r)$  
corresponds to a  slope-stable vector bundle, of $h$-slope 
$\mu(\cU_r) = 4(2c-b-a)-2$. In particular, $X$ is {\em (geometrically) $h$-Ulrich wild} with $Ur(X) = \mathbb N^*$, more precisely with no {\em slope-stable-$h$-Ulrich-rank gaps}.

\medskip

\noindent
(2) Let either $0 \leqslant a \leqslant 1$, $b \geqslant 2$ or $b \geqslant a \geqslant 1, \;(a,b) \neq (1,1)$ be integers and let $(X, h)$ be a threefold  scroll over $\FF_a$ as in \eqref{eq:Xe}. Then, for any $c \geqslant a+b+1$, $X$ is {\em (geometrically) $h$-Ulrich wild} with $Ur(X) = \mathbb N^*$, i.e. with no {\em indecomposable-$h$-Ulrich-rank gaps}. 
}

\bigskip

In the sequel, we will work over the complex field $\mathbb{C}$ and for standard terminology we always refer to \cite{H}. As a matter of notation, if $\mathcal P$ is either a {\em parameter space} of a flat family of geometric objects $\mathcal F$ or a {\em moduli space} on a varity $V$, $[\mathcal F]$ will denote the parameter (resp., moduli) point corresponding to (equivalence class of) $\mathcal F$.

\section{Proof of \textsc{Main Theorem}} This section is entirely devoted to the proof of \textsc{Main Theorem}. 

\begin{proof}[Proof of \textsc{Main Theorem}] (1) We first focus on the {\bf Inherited Case}, that deals with $L= \mathcal O_0(1,0,3c-b-1)$ and $L^U= \mathcal O_0(1, 2,c-1)$ as in \textsc{Theorem A}--(2) which, from $L(-h) = \varphi^*(\mathcal O_{\FF_0} (-1, 2c-b-1))$ and $L^U(-h) = \varphi^*(\mathcal O_{\FF_0} (1, -1))$, are $h$-Ulrich  {\em twisted pullbacks} from $\FF_0$ in the sense of \cite[Thm.2.12]{fa-fl-ma-pl}. 

Indeed $c_1(\mathcal E_0^b \otimes \mathcal O_{\FF_0}(1,c)) = \mathcal O_{\FF_0}(2,2c-b)$ is very-ample (by $c \geqslant b+1$), $\mathcal O_{\FF_0} (-1, 2c-b-1) \otimes \mathcal O_{\FF_0}(2,2c-b) = \mathcal O_{\FF_0}(1,4c-2b-1)$ and $ \mathcal O_{\FF_0} (1, -1) \otimes \mathcal O_{\FF_0}(2,2c-b) = \mathcal O_{\FF_0}(3,2c-b-1)$ are the only $\mathcal O_{\FF_0}(2,2c-b)$-Ulrich line bundles on $\FF_0$ (cf. \cite[Rem.\;3.3-(2)]{fa-fl-ma-pl}).

To establish the existence of rank-$r \geqslant 2$, $h$-Ulrich bundles on $X$, we use {\em recursive extensions} involving $(L, L^U)$. To do so, we ease notation by $L_1 := L^U$ and $L_2 := L$.

\medskip 

\noindent
{\bf Rank-two construction}: this is \cite[Thm.\,4.1-{\bf Case (2)}]{fa-fl-ma-pl} which, from 
$h^1(L_1-L_2)= 3(2c-b-1)$, $h^1(L_2 - L_1) = 2c-b+1$, can be used as a basic step for an {\em iterative process} and which deals with (deformations of) unsplitting extensions:
$$(g_2): 0 \to L_1 \to \cG_2 \to L_2 \to 0$$where $\cG_2$ turns out to be a $h$-Ulrich {\em twisted pullback} from $\mathbb{F}_0$; indeed from $(g_2)$, one has $\cG_2(-h) = \varphi^*(\cF_{2,\; \FF_0})$, with $\cF_{2,\, \FF_0}$ a rank-two vector bundle on $\FF_0$ fitting into 
$$(f_2):\;\; 0 \to \mathcal O_{\FF_0} (1, -1) \to \cF_{2,\; \FF_0} \to \mathcal O_{\FF_0} (-1, 2c-b-1)\to 0,$$so $\cF_{2,\; \FF_0}\otimes  \mathcal O_{\FF_0}(2,2c-b)$ is $\mathcal O_{\FF_0}(2,2c-b)$-Ulrich on $\FF_0$  (cf. \cite[(4.13)]{fa-fl-ma-pl}).

\medskip

\noindent
{\bf Iterations for $r \geqslant 3$}: for any $r \geqslant 3$, we recursively use:
$$(g_r): 0 \to \cG_{r-1} \to \cG_r \to L_{\epsilon_r} \to 0,\;\; \mbox{where} \;\; L_{\epsilon_r}=
    \begin{cases}
    L_1, & \mbox{if $r$ is odd}, \\
      L_2, & \mbox{if $r$ is even};
    \end{cases}$$This alternate use of line bundles on the right-side of $(g_r)$ and similar strategy as in \cite[Lemma 4.1-(iii)]{fa-fl2} allow, from coboundary map reasons, to find that: $$\dim(\Ext^1(L_{\epsilon_{r}},\cG_{r-1})) \geqslant   {\rm min} \{h^1(L_1-L_2) , \; h^1(L_2-L_1)\} = 2c-b+1 \geqslant b+3 \geqslant 3,$$the inequalities following from $c \geqslant b+1,\; b\geqslant 0$. Thus, at any step $r \geqslant 3$ unsplitting sequences $(g_r)$ actually exist, whose middle term $\cG_{r}$ is a rank-$r$, $h$-Ulrich bundle, of same $h$-slope as $\cG_2$, i.e. as in the statement, and whose Chern classes can be recursively computed by using $(g_r)$ and the expressions of $L_1= L^U$ and $L_2 = L$, i.e. 
{\small $$c_1(\cG_r) = c_1(\cG_{r-1}) + c_1 (L_{\epsilon_r}),\; c_2(\cG_r) = c_2(\cG_{r-1}) + 
c_1(\cG_{r-1}) \cdot c_1 (L_{\epsilon_r}), \; 
c_3(\cG_r) = c_3(\cG_{r-1}) + 
c_2(\cG_{r-1}) \cdot c_1 (L_{\epsilon_r}).$$}which give formulas as in \eqref{eq:chernLLU}. Morevover, since $\cG_r$ fits in $(g_r)$, where $\cG_{r-1}$ (by induction) and  $L_{\epsilon_{r}}$ are both $h$-Ulrich {\em twisted pullbacks}, the same holds true for $\cG_r$.

For $r \geqslant 3$ the bundle $\cG_{r-1}$ is strictly semistable thus, unlike the $r=2$ case,  to deduce simplicity for $\cG_r$ \cite[Lemma\,4.2]{c-h-g-s} does not apply. To circumvent this, we may replace $\cG_{r-1}$ with a slope-stable $\cU_{r-1}$ corresponding to a general point $[\cU_{r-1}] \in \mathcal{M}(r-1)$ of the modular component constructed by induction at the $(r-1)$-th step. This gives rise to an exact sequence:
$$(\widehat{g_r}): 0 \to \cU_{r-1} \to \widehat{\cG}_r \to L_{\epsilon_r} \to 0$$where 
$\dim(\Ext^1(L_{\epsilon_{r}},\cU_{r-1})) \geqslant \dim(\Ext^1(L_{\epsilon_{r}},\cG_{r-1})) -1 \geqslant b+2$ (cf. \cite[Lemma\,4.4]{fa-fl2} for similar reasoning). The latter inequality implies that $ [(\widehat{g_r})] \in \Ext^1(L_{\epsilon_{r}},\cU_{r-1}) $ general is certainly an unsplitting sequence. Morever, as both $\cU_{r-1}$ and $L_{\epsilon_r}$ are now slope-stable, non-isomorphic, with same $h$-slope, one can apply \cite[Lemma\,4.2]{c-h-g-s} so the middle term $\widehat{\cG}_r$ is $h$-Ulrich and simple. Furthermore, using induction, $(\widehat{g_r})$ and its dual sequence, one easily shows that $h^j(\widehat{\cG}_r \otimes \widehat{\cG}_r^\vee) = 0,\; j \geqslant 2$. Thus, from \cite[Prop.\,2.10]{c-h-g-s}, deformations of $\widehat{\cG}_r$ define a {\em flat, irreducible, smooth modular family} $\mathfrak{M}_r$ s.t. $\dim(\mathfrak{M}_r) =1-\chi(\widehat{\cG}_r \otimes \widehat{\cG}^{\vee}_r) = 1-\chi(\cG_r \otimes \cG_r^\vee)$ where, by $(g_r)$:
$$\chi(\cG_{r} \otimes \cG_{r}^{\vee}) = 
\begin{cases} 
\frac{(r^2-1)}{4} (2 - h^1(L_1-L_2)- h^1(L_2-L_1)) & \text{if $r$ is odd} \\ 
\frac{r^2}{4} (2 - h^1(L_1-L_2) - h^1(L_2-L_1)) & \text{if $r$ is even} 
\end{cases}$$which gives $\dim(\mathfrak{M}_r)$ as in \eqref{eq:dimLLU}. Similarly as for the $r=2$ case, slope stability of $[\cU_r] \in \mathfrak{M}_r$ general is verified via dimension counting: one has to show that families of $h$-Ulrich bundles that are not stable (i.e. arising as extensions via lower-rank $h$-Ulrich bundles--cf.\,e.g.\,\cite{b,c-h-g-s}) fill-up subsets of dimension strictly less than $\dim(\mathfrak{M}_r)$. These computations are much more involved than in the case $r=2$, so we refer the reader to  proofs of \cite[Lemmas\,4.6,\,4.7\,and\,Prop.\,4.8]{fa-fl2} for similar reasoning. Since $[\cU_r] \in \mathfrak{M}_r$ general turns out to be slope-stable, $\mathfrak{M}_r$ generically finitely dominates its GIT-quotient image $\mathcal M(r)$ which is the desired irreducible modular component as in the statement. 

We are left with proving {\em unirationality} of $\mathcal M(r)$ for any $r \geqslant 3$, as $\mathcal M(2)$ is {\em rational} by \cite[Thm.\,4.1-Case(2)]{fa-fl-ma-pl}. To do so, consider $(g_r)$ tensored by $h^{\vee}$ and apply $\varphi_*$ to get
{\small 
$$(f_r):\;\; 0 \to\cF_{r-1,\; \FF_0} \to \cF_{r,\; \FF_0} \to \mathcal L_{\epsilon_r,\; \FF_0} \to 0, \; {\rm where} \;\mathcal L_{\epsilon_r,\; \FF_0}:=
    \begin{cases}
    \mathcal O_{\FF_0}(1, -1), & \mbox{if $r$ is odd}, \\
      \mathcal O_{\FF_0}(-1, 2c-b-1), & \mbox{if $r$ is even}
    \end{cases}.$$}For any $r \geqslant 1$, we consider the associated bundles on $\FF_0$: $$\mathcal A_{\epsilon_r,\; \FF_0} :=  \mathcal L_{\epsilon_r,\; \FF_0} \otimes  \mathcal O_{\FF_0} (2, 2c-b) \; {\rm and} \; \cH_{r, \mathbb F_0} := \cF_{r,\; \FF_0} \otimes 
    \mathcal O_{\FF_0} (2, 2c-b), \;{\rm for} \;  r \geqslant 2.$$When $r \geqslant 2$, they are of Chern classes $ \overline{c}_1:= c_1 (\cH_{r, \mathbb F_0})$ and $\overline{c}_2:= c_2 (\cH_{r, \mathbb F_0})$ s.t. $$\overline{c}_1= c_1(\cF_{r,\; \FF_0}) \otimes \mathcal O_{\FF_0}(2r, r(2c-b)), \; \overline{c}_2= c_2(\cF_{r,\; \FF_0}) + (r-1) c_1(\cF_{r,\; \FF_0}) \cdot \mathcal O_{\FF_0} (2, 2c-b) + 2r(r-1)(2c-b),$$where $c_i(\cF_{r,\; \FF_0})$, $1 \leqslant i \leqslant 2$, can be recursively computed via $(f_r)$. They are moreover $ \mathcal O_{\FF_0} (2, 2c-b)$-Ulrich on $\FF_0$, as one can easily show by recursive vanishings obtained via the use of $(f_r)$ (cf. also the {\em admissible Ulrich pairs} as in \cite{ant}). 

 As above for $\widehat{\cG}_r$ on $X$, flat deformations of bundles $\cH_r$ give rise to an irreducible, smooth modular family $\mathfrak{M}_{r,\FF_0}$ whose projection to the moduli space $\mathcal M_{\FF_0}(r; \overline{c}_1\, \overline{c}_2)$ of $ \mathcal O_{\FF_0} (2, 2c-b)$-Ulrich bundles of rank $r$ and Chern classes $\overline{c}_1\, \overline{c}_2$ as above is generically finite, so it gives rise to an irreducible component $\mathcal M_{\FF_0}(r)$ which turns out to be generically smooth, whose general point $[\cH_r] $ is slope-stable, 
and which is of dimension $\dim(\mathcal M_{\FF_0}(r)) = 1 - \chi (\cH_{r, \mathbb F_0} \otimes \cH^{\vee}_{r, \mathbb F_0})  = 1 -  \chi (\cF_{r, \mathbb F_0} \otimes \cF^{\vee}_{r, \mathbb F_0})$, the latter recursively computed via $(f_r)$ so to have: 
\begin{eqnarray}\label{eq:dim2}
\dim(\mathcal M_{\FF_0}(r)) = 
		\begin{cases} (r^2 -1)(2c-b-1), & \mbox{if $r\geqslant 3$ is odd}, \\
			 r^2(2c-b-1)+1, & \mbox{if $r$ is even}.
    \end{cases}
    \end{eqnarray}By \cite[Thm.\,1.1]{ant}, $\cH_r $  necessarily  fits into the exact sequence

\begin{eqnarray}\label{eq:coker}
0 \to \mathcal O_{\FF_0}(1, 2c-b -1)^{\oplus\,\gamma} \stackrel{\phi}{\longrightarrow}   \mathcal O_{\FF_0}(1, 2c-b)^{\oplus \delta} \oplus \mathcal O_{\FF_0}(2, 2c-b-1)^{\oplus \tau} \to\cH_r  \to 0                            
 \end{eqnarray}where $$\begin{cases}
     \gamma=\frac{(2c-b)(r-1)}{2}+1,\; \delta= \frac{(r-1)(b-2c)}{2}, \tau=r+1 & \mbox{if $r$ is odd}, \\
      \gamma= \frac{(b-2c)r}{2},\;\delta= \frac{(b-2c)r}{2},\;\tau=r & \mbox{if $r$ is even}.
    \end{cases}$$Namely $\cH_r $ arises as the cokernel of $[\phi] \in \rm{Hom}_{\FF_e} ( \sA, \sB)$, where $$\sA :=  \mathcal O_{\FF_0}(1, 2c-b -1)^{\oplus \gamma} \;\; {\rm and} \;\; \sB :=  \mathcal O_{\FF_0}(1, 2c-b)^{\oplus \delta} \oplus \mathcal O_{\FF_0}(2, 2c-b-1)^{\oplus \tau},$$with $\gamma, \delta, \tau$  as above. 
    
    On  the other hand, by  \cite[Thm.\;1.3]{ant}, for $[\phi_{gen}] \in \rm{Hom}_{\FF_e} ( \sA, \sB)$ general, ${\rm coker}(\phi_{gen})$ is of rank $r$, $\mathcal O_{\FF_0}(2, 2c-b)$-Ulrich on $\FF_0$, with Chern classes 
    $c_1({\rm coker}(\phi_{gen})) = \overline{c}_1$ and $c_2({\rm coker}(\phi_{gen})) = \overline{c}_2$. Since $ \sA, \sB$ are uniquely  determined  by the integers $r,c,b$, as it occurs for $ \overline{c}_i$, $1 \leqslant i \leqslant 2$, and since $\rm{Hom}_{\FF_0} ( \sA, \sB)$ is irreducible, then $\mathcal M_{\FF_0}(r)$ is the only irreducible component i.e. $\mathcal M_{\FF_0}(r; \overline{c}_1\, \overline{c}_2) = \mathcal M_{\FF_0}(r)$ is irreducible and unirational, being dominated by $\rm{Hom}_{\FF_0} ( \sA, \sB)$.  Since $\cF_{r,\; \FF_0}:= \cH_{r, \FF_0} \otimes \mathcal O_{\FF_0} (- 2, -(2c-b))$ and 
   $\widehat{\cG}_r = h \otimes \varphi^*(\cF_{r,\; \FF_0})$, we therefore have a natural morphism 
$$\mathcal M_{\FF_0}(r)\stackrel{\psi}{\longrightarrow} \mathfrak{M}_r, \;\; 
[\cH_{r,\; \FF_0}] \mapsto [\psi(\cH_{r,\; \FF_0})]:= [ h  \otimes  \varphi^*(\cF_{r,\; \FF_0})]= [\widehat{\cG_r}].$$Taking any section $\sigma:\FF_0\rightarrow X$ we have 
$id_{\FF_0}^*=(\varphi\circ\sigma)^*=\sigma^*\circ\varphi^*$, thus $\psi$ is injective. Since $\mathfrak{M}_r$ and $\mathcal M_{\FF_0}(r)$ have the same dimension, by \eqref{eq:dimLLU} and \eqref{eq:dim2}, then (being injective) $\psi$ is also dominant therefore birational. Unirationality of $\mathcal M_{\FF_0}(r)$ implies therefore that $\mathfrak{M}_r$ is unirational; since moreover $[\cU_r] \in \mathfrak{M}_r$ general has been proved to be slope-stable, $\mathfrak{M}_r$ generically finitely dominates the modular component $\mathcal M(r)$ via the GIT-quotient map so $\mathcal M(r)$ is unirational too.

Finally, from \textsc{Theorem A}-(2) and from above, $X$ is s.t. $uc_h(X)=1$ and it supports modular components $\mathcal{M}(r)$ of slope-stable $h$-Ulrich bundles for any rank $r \geqslant 2$, so in particular $X$ is {\em (geometrically) $h$-Ulrich wild} with $Ur(X) = \mathbb N^*$, i.e. more precisely with no slope-stable $h$-Ulrich-rank gaps.

\bigskip

\bigskip

\noindent
(2) This is the {\bf Obstructed Case}, dealing with $N = \mathcal O_a(2, 0,2c-a-1)$ and $N^U = \mathcal O_a(0, 2,2c-b-1)$, for any  $b \geqslant a \geqslant 0$ s.t. $(a,b) \neq (0,0),\,(0,1)\,(1,1)$ (these 3 latter cases being {\em unobstructed} and already discussed in \cite[Thm.\,5.6]{fa-fl-ma-pl}). As for Case (1), we ease notation by taking $N_1:= N^U$, $N_2:=N$ and $$N_{\epsilon_r}=
    \begin{cases}
    N_1, & \mbox{if $r$ is odd}, \\
      N_2, & \mbox{if $r$ is even}.
    \end{cases}$$From the proofs of \cite[Thms.\,4.1-($1_b$)\,and\,4.2-(jj),\,(jjj)]{fa-fl-ma-pl}  one has: 
\begin{eqnarray}\label{eq:com}
h^1(N_2-N_1)=b+2,\;h^1(N_1-N_2)=a+2,\;h^2(N_2-N_1)=b-1,\\
\nonumber h^2(N_1-N_2)= {\rm min} \{0, a-1\}, \;h^j(N_2-N_1) = h^j(N_1-N_2)=0, \; j=0,3.
 \end{eqnarray} Thus, from e.g. $\dim({\rm Ext}^1 (N_2,N_1))=h^1(N_1-N_2)= a+2 \geqslant 2$, any non-zero $[(g_2)] \in {\rm Ext}^1 (N_2,N_1)$ gives rise to an unsplitting sequence 
$$(g_2):\; \; 0 \to N_1\stackrel{\iota_2}{\longrightarrow}\cG_2 \stackrel{p_2}{\longrightarrow}  N_2 \to 0,$$whose middle-term $\cG_2$ is strictly semistable, $h$-Ulrich, of  rank two and of the same $h$-slope of $N_1$ and $N_2$, i.e. 
$\mu = 4(2c-b-a-2)-2$. Similarly as in Case (1), for $r \geqslant 3$, we then proceed by taking {\em alternating extensions} $$(g_r):\;\; 0 \to \cG_{r-1} \stackrel{\iota_r}{\longrightarrow} \cG_r \stackrel{p_r}{\longrightarrow} N_{\epsilon_r} \to 0,$$as, from coboundary map reason, the alternate use of line bundles on the right-side prevents one to get, after finitely many steps, a zero-dimensional extension spaces, i.e.:

\begin{claim}\label{cl:ext} For any $r \geqslant 3$, one has 
$\dim({\rm Ext}^1 (\cG_{r-1},N_{\epsilon_r})) \geqslant \begin{cases}
    r-1, & \mbox{if $r \geqslant 3$ is odd} \\
      r+1, & \mbox{if $r \geqslant 4$ is even}
    \end{cases}$
\end{claim}
\begin{proof}[Proof of Claim \ref{cl:ext}] We use $(g_{r-1}): 0 \to \cG_{r-2} \to \cG_{r-1} \to  N_{\epsilon_{r-1}} \to 0$, inductively constructed, tensored by $N_{\epsilon_r}^{\vee}$.  Since $\epsilon_{r-1}$ and $\epsilon_{r}$ have different parity, by induction we get 
 $h^3(\cG_{r-1} \otimes  N_{\epsilon_r}^{\vee}) = 0$ so, in particular, 
\begin{eqnarray}\label{eq:boundutile}
h^1(\cG_{r-1} \otimes  N_{\epsilon_r}^{\vee})=  h^0(\cG_{r-1} \otimes  N_{\epsilon_r}^{\vee})+h^2(\cG_{r-1} \otimes  N_{\epsilon_r}^{\vee})-\chi(\cG_{r-1} \otimes  N_{\epsilon_r}^{\vee})\geqslant -\chi(\cG_{r-1} \otimes  N_{\epsilon_r}^{\vee})
\end{eqnarray} Since, for any $r \geqslant 3$, $\cG_{r-1}$ is strictly semistable, it is $S$-equivalent to $ N_1^{\oplus r_{N_1}}  \oplus N_2^{\oplus r_{N_2}}$, where 
$r_{N_1}$  (resp., $r_{N_2}$) denotes the number of copies of $N_1$ (resp., $N_2$) used in the successive extension, where $r_{N_1}+r_{N_2}= {\rm rk}(\cG_{r-1})= r-1$. Thus
\begin{eqnarray}\label{eq:aiuto1} 
\chi(\cG_{r-1} \otimes  N_{\epsilon_r}^{\vee})=r_{N_1} \, \chi(N_1\otimes  N_{\epsilon_r}^{\vee})+r_{N_2}\, \chi(N_2\otimes  N_{\epsilon_r}^{\vee}).
\end{eqnarray} From \eqref{eq:com}, $\chi(N_1- N_2)= \chi(N_2-N_1) =-3$ whereas $\chi(N_{\epsilon_r}-N_{\epsilon_r})=\chi(\mathcal O_{X})=1$. 

Therefore if $\epsilon_r=1$, i.e. $r$ is odd so $r-1$ is even, from \eqref{eq:aiuto1} and $r_{N_1}+r_{N_2}=r-1$, we get 
$\chi(\cG_{r-1} \otimes  N_{1}^{\vee})=r_{N_1}\chi(\mathcal O_{X})+r_{N_2} \; \chi(N_2\otimes  N_{1}^{\vee})= r_{N_1}-3\,r_{N_2}=(r-1)-4\,r_{N_2}$. If otherwise $\epsilon_r=2$, i.e. $r$ is even so $(r-1)$ is odd, similarly we get $\chi(\cG_{r-1} \otimes  N_{2}^{\vee})=(r-1)-4\,r_{N_1}$. Since we are taking {\em alternating extensions}, when $r$ is odd,  $N_1$ is used as many times as $N_2$ plus one, whereas when $r$ is even $N_1$ is used as many times as $N_2$. Therefore, since $r_{N_1}+r_{N_2}=r-1$, one has  
$r_{N_1}= \Big\lceil{\frac{r-1}{2}}\Big\rceil \; \mbox{and}  \; r_{N_2}= \Big\lfloor{\frac{r-1}{2}}\Big\rfloor$. 

Hence, if $r = 2k+1$ is odd, $r_{N_1} =r_{N_2}=k$ and $\chi(\cG_{r-1} \otimes  N_{1}^{\vee})=\chi(\cG_{2k} \otimes N_{1}^{\vee})=(r-1)-4r_{N_2} = 2k-4(k)=-2k=-(r-1)$ whence $h^1(\cG_{r-1} \otimes  N_{1}^{\vee})\geqslant r-1$. If otherwise $r=2k \geqslant 4$ is even, $r_{N_1} = k$, $r_{N_2}=k-1$ and $\chi(\cG_{r-1} \otimes  N_{2}^{\vee})=\chi(\cG_{2k-1} \otimes N_{2}^{\vee})=(r-1)-4r_{N_1} = (2k-1)-4k=-2k-1= -(r+1)$, thus $h^1(\cG_{r-1} \otimes  N_{2}^{\vee})\geqslant r+1$, concluding the proof of the claim. 
\end{proof} 

From the above claim, at any $r \geqslant 3$ step we certainly have unsplitting sequences $(g_r)$, whose middle-term $\cG_r$ is of rank $r$, $h$-Ulrich, strictly semistable, of the same $h$-slope $\mu = 4(2c-b-a-2)-2$ as $\cG_2$, $N_1$ and $N_2$.

\begin{claim}\label{cl:unsp} For any $r \geqslant 2$, $[(g_r)] \in {\rm Ext}^1 (\cG_{r-1},N_{\epsilon_r})$ general gives rise to an indecomposable  
middle-term $ \cG_r$. 
\end{claim} 
\begin{proof}[Proof of Claim \ref{cl:unsp}] For $r=2$, $\cG_{r-1} = \cG_1 := N_1$ and $N_{\epsilon_2} = N_2$ and the proof is similar to that in Case (1); indeed,  being $N_1$ and $N_2$ non-isomorphic, slope stable $h$-Ulrich line bundles with the same $h$-slope, by \cite[Lemma\,4.2]{c-h-g-s} and by $\dim({\rm Ext^1}(N_2,N_1)) = h^1(N_1-N_2) = a+2 \geqslant 2 >0$, $\cG_2$ is  simple, in particular indecomposable.

We therefore take $r \geqslant 3$ and proceed by induction, from which we may assume $\cG_{r-1}$ to be indecomposable; moreover, from Claim \ref{cl:ext}, $[(g_r)] \in {\rm Ext}^1 (\cG_{r-1},N_{\epsilon_r})$ general gives rise to an unsplitting sequence. Assume by contradiction that the associated middle-term $\cG_r$ is decomposable, i.e. of the form $\cG_r = A_k \oplus B_{r-k}$, for some vector bundles s.t. $1 \leqslant k:={\rm rk}(A_k), \; r-k: = {\rm rk}(B_{r-k}) < r$. 

Considering $(g_r)$ above, since $ N_{\epsilon_r}$ is a line bundle, at least one of the two summands of $\cG_r$ must surject onto $ N_{\epsilon_r}$; assume this is $A_k$. If  $k= {\rm rk} (A_k) =1$ then $A_k \cong N_{\epsilon_r}$, being both line bundles, so we would have 
\begin{eqnarray*}
    \xymatrix@R=0.1pc{ N_{\epsilon_r} \cong A_k  \ar@{^{(}->}[r]\ar[r]^{\;\; \; \;j_{A_k}}& \cG_r \ar@{->>}[r]^{p_r} & N_{\epsilon_r}\\}
 \end{eqnarray*}where $j_{A_k}$ denotes the natural inclusion of $A_k$ as a summand of $\cG_{r}$; so there would exist $\alpha \in \mathbb C^{*}$ s.t. $\alpha (p_r\circ j_{A_K})=Id_{N_{\epsilon_r}}$ contradicting that $(g_r)$ is unsplitting. Assume therefore $k = {\rm rk} (A_k) > 1$ and consider the following diagram: 
\begin{eqnarray*}
    \xymatrix{  &&0\ar[d] &\\
     &&B_{r-k}\ar[d]^{j_{B_{r-k}}}\ar@{-->}[rd]^{p_{r_{|B_{r-k}}}} &\\
    0 \ar[r]  &\cG_{r-1} \ar[r] ^{\hspace{-9mm}\iota_r} & \cG_r = A_k \oplus B_{r-k}  \ar@{->>}[d]^{\pi_{A_k}}\ar@{->>}[r]^{\hspace{9mm}p_r} & N_{\epsilon_r}  \ar@{=}[d]^{Id_{N_{\epsilon_r}}}\\
 0 \ar[r]  &\cL_{k-1}:= Ker({p_r}_{|A_k}) \ar[r] &  A_k \ar@{->>}[r]^{{p_r}_{|A_k}}  \ar[d] & N_{\epsilon_r} \\
 &&0&}
 \end{eqnarray*}If ${p_r}_{|B_{r-k}}=0$, then $B_{r-k} \subseteq Ker(p_r)=\cG_{r-1}$ and $p_r= {p_r}_{|A_k} \circ \pi_{A_k}$, which gives contradiction because 
$Ker(p_r) = \cG_{r-1}$ is indecomposable by inductive step, whereas $Ker({p_r}_{|A_k} \circ \pi_{A_k}) = \cL_{k-1} \oplus B_{r-k}$. If otherwise ${p_r}_{|B_{r-k}}\neq 0$, then $Im({p_r}_{|B_{r-k}}) \subseteq N_{\epsilon_r}$ so there exists an effective divisor $D \geqslant 0$ such that $Im({p_r}_{|B_{r-k}}) \cong 
N_{\epsilon_r}(-D)$. Now 
$\cG_{r-1} = Ker(p_r) =\left\{(a,b) \in \cG_r =   A_k \oplus B_{r-k} \; | \; p_r(b) = - p_r(a) \in N_{\epsilon_r} \right\}$; since $\pi_{A_k}$ and ${p_r}_{|A_k}$ are by assumption both surjective onto $N_{\epsilon_r}$ and since from the diagram $p_r= {p_r}_{|A_k} \circ \pi_{A_k}$, then for any $b \in B_{r-k}$ s.t. $p_r(b) \in N_{\epsilon_r}(-D)$, there always exist $a \in A_k$ s.t. 
$- {p_r}_{|A_k} (a) = {p_r}_{|B_{r-k}}(b) \in N_{\epsilon_r}(-D)$, i.e. $\cG_{r-1} = Ker(p_r) \cong \mathcal L'_{k-1} \oplus B'_{r-k}$ for some $\mathcal L'_{k-1} \subset A_k$ and $B'_{r-k} \subseteq B_{r-k}$, once again contradicting 
that $\cG_{r-1}$ is indecomposable by induction, concluding the proof of the claim.  
\end{proof} We are left with the following

\begin{claim}\label{cl:noiso} For any $r \geqslant 2$, given $[(g_r)] \neq [(\widehat{g_r}] \in {\rm Ext}^1 (\cG_{r-1},N_{\epsilon_r})$ general points, then the corresponding middle-term vector bundles $\cG_r$ and $\widehat{\cG}_r$ are not isomorphic. 
\end{claim}

\begin{proof}[Proof of Claim \ref{cl:noiso}] From \cite{S}, for any $r \geqslant 2$ the group ${\rm Aut}(\cG_{r-1}) \times {\rm Aut}(N_{\epsilon_r})$ acts by (push-out, pull-back) on ${\rm Ext}^1 (\cG_{r-1},N_{\epsilon_r})$: for any $(\alpha, \; \beta) \in {\rm Aut}(\cG_{r-1}) \times {\rm Aut}(N_{\epsilon_r})$ and any 
$[(g_r)] \in {\rm Ext}^1 (\cG_{r-1},N_{\epsilon_r})$ one has:  
$$(\alpha,\beta) ([(g_r)]):= \alpha_* \circ [(g_r)] \circ \beta^* \in {\rm Ext}^1 (\cG_{r-1},N_{\epsilon_r})$$and the corresponding middle-terms are isomorphic as rank-$r$ vector bundles. 

Conversely, assume to have two extensions $[(g_r)] \neq  [(\widehat{g_r}] \in {\rm Ext}^1 (\cG_{r-1},N_{\epsilon_r})$ for which there exists an isomorphism $\psi$ of vector bundles 
between the middle terms, namely: 
\[
\begin{array}{rcccl} 
(\widehat{g_r}):\;\; 0 \to \cG_{r-1} & \stackrel{\widehat{\iota_r}}{\longrightarrow} & \widehat{\cG}_r & \stackrel{\widehat{p_r}}{\longrightarrow} & N_{\epsilon_r} \to 0\\
 & & \;\;\; \downarrow^{\cong_\psi} & & \\
(g_r): \;\; 0 \to \cG_{r-1} & \stackrel{\iota_r}{\longrightarrow} & \cG_r & \stackrel{p_r}{\longrightarrow} & N_{\epsilon_r} \to 0\\
\end{array}
\]$\widehat{\cG}_r$ (resp., $\cG_r$) is strictly semistable and 
$\widehat{\iota_r}(\cG_{r-1})$ (resp., $\iota_r(\cG_r)$) is a {\em maximal destabilizing} sub-vector bundle, since of rank $r-1$. Therefore $\psi$ must send 
$ \widehat{\iota_r}(\cG_{r-1})$ to $\iota_r(\cG_r)$. This implies there exists $\alpha \in 
{\rm Aut} (\cG_{r-1})$ such that $\psi \circ \widehat{\iota_r} = \psi_{|_{\cG_{r-1}}} = \iota_r \circ \alpha$. Since $\frac{\widehat{\cG}_r}{\widehat{\iota_r}(\cG_{r-1})}\cong N_{\epsilon_r} \cong \frac{\cG_r}{\iota_r(\cG_{r-1})}$, then $\psi$ induces an automorphism $\beta:= \overline{\psi} \in {\rm Aut}(N_{\epsilon_r}) \cong \mathbb C^*$ which let the previous diagram commute, namely $\overline{\psi} \circ \widehat{p_r} = p_r \circ \psi$; in other words 
$[(\widehat{g_r})]$ and $[(g_r)]$ are in the same $\left( {\rm Aut}(\cG_{r-1}) \times {\rm Aut}(N_{\epsilon_r}) \right)$-orbit. 

To prove the Claim, we therefore need to show that the {\em maximal} possible dimension for the $\left({\rm Aut}(\cG_{r-1}) \times {\rm Aut}(N_{\epsilon_r})\right)$-orbits is strictly less than $\dim( {\rm Ext}^1 (\cG_{r-1},N_{\epsilon_r})$. This will be done by induction. For $r=2$, recall that $\cG_{r-1} = \cG_1: = N_1$, $N_{\epsilon_2} = N_2$, so $ {\rm Aut}(N_1) \times {\rm Aut}(N_{2})$ has dimension $2$ whereas, from above, $\dim({\rm Ext}^1(N_2, N_1)) = h^1(N_1-N_2) = a+2$, strictly bigger than $2$ unless $a=0$. In this latter case, from the assumptions on $(0,b)$, one must have $b \geqslant 2$ so we change  the order of $N_1$ and $N_2$ considering $ {\rm Ext}^1(N_1, N_2) \cong H^1(N_2-N_1)$, whose elements gives rise to rank-two, $h$-Ulrich, simple vector bundles $\cG'_2$ of same Chern classes and slope as $\cG_2$, to which Claim \ref{cl:unsp} applies verbatim, i.e. $\cG'_2$ general is indecomposable. With this permutation, $\dim({\rm Ext}^1(N_1, N_2)) = h^1(N_2-N_1) = b+2 \geqslant 4$ and we are done also for $a=0$. Observe further that, taking the dual sequence of $(g_2)$ tensored by $N_2$, one gets
$0  \to N_2\otimes N^{\vee}_2 \cong \mathcal O_X \to N_2 \otimes \cG_2^{\vee} \to N_2 \otimes N^{\vee}_1 \to 0$, from which one deduces 
\begin{equation}\label{eq:aiuto3}
{\rm Hom}(\cG_2,N_2) \cong H^0(N_2 \otimes \cG_2^{\vee}) \cong H^0(\mathcal O_X) \cong {\rm Hom}(N_2, N_2) \cong \mathbb C,
\end{equation}which will be used below for the inductive procedure. 

Let us assume now $r \geqslant 3$, for which one has 
\begin{equation}\label{eq:aiuto(a)}
{\rm Hom} (N_{\epsilon_r}, \cG_{r-1}) \cong  {\rm Hom} (N_{\epsilon_r}, \cG_{r-2}) \; {\rm whereas} \; {\rm Hom} (\cG_{r-1}, N_{\epsilon_r} ) = (0), 
\end{equation}where $\cG_1 := N_1$ when $r-2=1$. The first isomorphism easily follows from $(g_{r-1})$ tensored by $N^{\vee}_{\epsilon_r}$, taking into account that $N_{\epsilon_r}$ and $N_{\epsilon_{r-1}}$ have different parity so $h^0(N_{\epsilon_{r-1}} \otimes N^{\vee}_{\epsilon_{r}}) = 0$. Concerning the second equality, it follows from the dual sequence of $(g_{r-1})$ tensored by $N_{\epsilon_r} = N_{\epsilon_{r-2}}$, where $h^0(N^{\vee}_{\epsilon_{r-1}} \otimes N_{\epsilon_{r-2}}) = 0$ because of different parity; indeed if $r=3$, then $ H^0( \cG^{\vee}_{r-2} \otimes N_{\epsilon_{r-2}} ) \cong {\rm Hom}(N_1, N_1) \cong \mathbb C$; whereas for $r = 4$, $ H^0( \cG^{\vee}_{r-2} \otimes N_{\epsilon_{r-2}} ) \cong {\rm Hom} (\cG_2, N_2) \cong {\rm Hom}(N_2, N_2) \cong \mathbb C$ as from \eqref{eq:aiuto3}; by induction and iterating the reasoning, from the previous exact sequence one has 
{\small $$H^0(\cG^{\vee}_{r-2} \otimes N_{\epsilon_{r-2}} ) \cong {\rm Hom}(N_{\epsilon_{r-2}}, N_{\epsilon_{r-2}}) \cong \mathbb C \stackrel{\partial}{\longrightarrow} H^1(N^{\vee}_{\epsilon_{r-1}} \otimes N_{\epsilon_{r-2}}) \cong {\rm Ext}^1(N_{\epsilon_{r-1}}, N_{\epsilon_{r-2}}),$$}where $\partial$ is the coboundary map, which is injective by non-splitting sequence. This injectivity implies 
${\rm Hom} (\cG_{r-1}, N_{\epsilon_r} ) = H^0( \cG^{\vee}_{r-1} \otimes N_{\epsilon_{r-2}}) \cong H^0(N^{\vee}_{\epsilon_{r-1}} \otimes N_{\epsilon_{r-2}}) = (0)$, the latter equality following by different parity, proving the second equality in \eqref{eq:aiuto(a)}.

Now $(g_r)$ tensored by $\cG^{\vee}_{r-1}$ gives 
$0 \to {\rm End}(\cG_{r-1}) \to  {\rm Hom}(\cG_{r-1}, \cG_r) \to {\rm Hom} (\cG_{r-1}, N_{\epsilon_r})$ so, from the vanishing in \eqref{eq:aiuto(a)}, we get 
\begin{equation}\label{eq:aiuto(e)}
{\rm End}(\cG_{r-1}) \cong {\rm Hom}(\cG_{r-1}, \cG_r)
\end{equation} Taking now the dual sequence of $(g_r)$ tensored by $\cG_{r}$ and passing to cohomology gives 
$$0 \to  {\rm Hom}(N_{\epsilon_r}, \cG_r) \to {\rm End}(\cG_r) \to {\rm Hom}(\cG_{r-1}, \cG_r) 
\cong {\rm End}(\cG_{r-1}) \to \cdots,$$the isomorphism on the right-side following from \eqref{eq:aiuto(e)}.  In other words, for any $r \geqslant 3$, one has that  
\begin{equation}\label{eq:ineq}
1 \leqslant \dim({\rm End}(\cG_r)) \leqslant \dim({\rm Hom}(N_{\epsilon_r}, \cG_r)) + \dim ({\rm End}(\cG_{r-1})).
\end{equation}We therefore need to compute $\dim({\rm Hom}(N_{\epsilon_r}, \cG_r))$. Observe that, for any $r \geqslant 2$, one has also 
\begin{equation}\label{eq:aiuto(c)}
{\rm Hom}(N_{\epsilon_r}, \cG_r) \cong {\rm Hom}(N_{\epsilon_r}, \cG_{r-1});
\end{equation}indeed, considering $(g_r)$ then, for any $f \in {\rm Hom}(N_{\epsilon_r}, \cG_r)$, one must have $p_r \circ f = 0$ otherwise $(g_r)$ would be a splitting sequence, against assumptions. In other words, $Im(f) \subseteq Ker(p_r) = \cG_{r-1})$, i.e. $f \in  {\rm Hom}(N_{\epsilon_r}, \cG_{r-1})$. Thus, from \eqref{eq:aiuto(c)}, we are reduced to compute $\dim({\rm Hom}(N_{\epsilon_r}, \cG_{r-1}))$ for any $r \geqslant 2$. By recursively applying the first isomorphism in \eqref{eq:aiuto(a)}, if $r = 2k+1$ is odd, then $N_{\epsilon_{r}} = N_1$ so
{\small \begin{equation}\label{eq:aiuto(b)odd}
{\rm Hom}(N_{\epsilon_r}, \cG_{r-1}) = {\rm Hom} (N_1, \cG_{2k}) \cong  {\rm Hom} (N_1, \cG_{2k-1}) \cong \cdots \cong {\rm Hom} (N_1, N_1) \cong \mathbb C. 
\end{equation}}When otherwise $r = 2k$ is even, then $N_{\epsilon_{r}} = N_2$ and, from recursive application of the first isomorphism in \eqref{eq:aiuto(a)}, we get
{\small \begin{equation}\label{eq:aiuto(b)even}
{\rm Hom}(N_{\epsilon_r}, \cG_{r-1}) = {\rm Hom} (N_2, \cG_{2k-1}) \cong  {\rm Hom} (N_2, \cG_{2k-2}) \cong \cdots \cong {\rm Hom} (N_2, N_1) = (0). 
\end{equation}}Putting together \eqref{eq:aiuto(c)}, \eqref{eq:aiuto(b)odd} and 
\eqref{eq:aiuto(b)even}, we get 

$$\dim({\rm Hom}(N_{\epsilon_r}, \cG_r)) = \begin{cases}
    1  & \mbox{if $r$ is odd}, \\
     0 & \mbox{if $r$ is even}
    \end{cases}$$ which, together with \eqref{eq:ineq}, gives
\begin{equation}\label{eq:aiuto(f)}
1 \leqslant \dim({\rm End}(\cG_r)) \leqslant \begin{cases}
    \dim({\rm End}(\cG_{r-1})) +1 & \mbox{if $r\geqslant 3$ is odd}, \\
     \dim({\rm End}(\cG_{r-1})) & \mbox{if $r$ is even}.
    \end{cases}
\end{equation}Since, for $r=2$, $\dim({\rm End}(\cG_2)) =1$ because $\cG_2$ is simple, one deduces from \eqref{eq:aiuto(f)} that, for any $r \geqslant 2$, one has  $1 \leqslant \dim({\rm End}(\cG_r)) = \dim ({\rm Aut}(\cG_r)) \leqslant \lceil \frac{r}{2} \rceil$. 

We deduce therefore that $$\dim\left({\rm Aut}(\cG_{r-1}) \times {\rm Aut}(N_{\epsilon_r})\right) \leqslant\lceil \frac{r-1}{2} \rceil + 1.$$

Hence, if $r = 2k \geqslant 4$ is even, 
$\dim\left({\rm Aut}(\cG_{r-1}) \times {\rm Aut}(N_{\epsilon_r})\right) \leqslant (k-1) + 1 = k$ whereas, from Claim \ref{cl:ext}, $\dim({\rm Ext}^1(\cG_{r-1}, N_{\epsilon_r})) \geqslant 2k+1$. 

When otherwise $r= 2k+1 \geqslant 3$ is odd,  $\dim\left({\rm Aut}(\cG_{r-1}) \times {\rm Aut}(N_{\epsilon_r})\right) \leqslant k + 1$ whereas from above we have $\dim({\rm Ext}^1(\cG_{r-1}, N_{\epsilon_r})) \geqslant 2k$. Thus, except for $k=1$ i.e. $r=3$, the latter is strictly bigger than $k+1$. We are left with the case $r=3$, where we can improve dimensional counting as in Claim \ref{cl:ext}; indeed, from \eqref{eq:boundutile}, 
$\dim({\rm Ext}^1(\cG_2, N_1)) = h^1(\cG_{2} \otimes  N_1^{\vee})=  h^0(\cG_2 \otimes  N_1^{\vee})+h^2(\cG_2 \otimes  N_{1}^{\vee})-\chi(\cG_{2} \otimes  N_{1}^{\vee})$, where 
$-\chi(\cG_{2} \otimes  N_{1}^{\vee})\geqslant 2$ as computed in the proof of Claim \ref{cl:ext} but also $h^0(\cG_2 \otimes  N_1^{\vee}) = 1$ as it follows from $(g_2)$ tensored with $ N_1^{\vee}$, i.e. $\dim({\rm Ext}^1(\cG_2, N_1))= 1 + h^2(\cG_2 \otimes  N_{1}^{\vee}) - \chi(\cG_{2} \otimes  N_{1}^{\vee}) \geqslant 1 + 2 = 3$, whereas $\dim\left({\rm Aut}(\cG_{2}) \times {\rm Aut}(N_{1})\right) = 2$, completely proving the claim.  
\end{proof}

Since we have produced on $X$ families of indecomposable, pairwise non-isomorphic, $h$-Ulrich bundles of rank $r \geqslant 2$ and from $uc_h(X) =1$ in \textsc{Theorem A}-(2), $X$ is {\em $h$-Ulrich wild} with $Ur(X) = \mathbb N^*$, i.e. there are no indecomposable-$h$-Ulrich rank gaps (even if there are {\em cohomological obstructions} to produce positive dimensional modular components in any rank). This completes the proof of part (2).  
\end{proof}

\end{document}